\documentclass[a4paper, 11pt]{amsart}

\usepackage[T1]{fontenc}
\usepackage[utf8]{inputenc}

\usepackage{amssymb}

\makeatletter
\@namedef{subjclassname@2010}{%
  \textup{2010} Mathematics Subject Classification}
\makeatother

\allowdisplaybreaks[2]

\newtheorem{theorem}{Theorem}[section]

\newtheorem{lemma}[theorem]{Lemma}
\newtheorem{corollary}[theorem]{Corollary}
\theoremstyle{definition}
\newtheorem{remark}[theorem]{Remark}
\newtheorem{definition}[theorem]{Definition}
\newtheorem{example}[theorem]{Example}

\numberwithin{equation}{section}
\begin{document}

\title[On, Around, and Beyond Frobenius' Theorem]{On, Around, and Beyond Frobenius' Theorem on Division Algebras}
\author{Matej Bre\v sar}
\author{Victor S. Shulman}

\address{M. Bre\v sar,  Faculty of Mathematics and Physics,  University of Ljubljana,
 and Faculty of Natural Sciences and Mathematics, University 
of Maribor, Slovenia}
 \email{matej.bresar@fmf.uni-lj.si}

\address{V. Shulman,  Department of Mathematics, 
Vologda State University, 
Vo\-logda, Russia}

\begin{abstract}
 Frobenius' Theorem states that  the algebra of quaternions $\mathbb H$ is, besides the fields of real and complex numbers, the only finite-dimensional real division algebra. We first give a short elementary proof of this theorem, then characterize finite-dimensional real algebras that contain either a copy of $\mathbb C$, a copy of $\mathbb H$, or a pair of anticommuting invertible elements through the dimensions of their (left) ideals, and finally consider the problem of lifting algebraic elements  modulo ideals.
\end{abstract}

\email{shulman.victor80@gmail.com}

\thanks{The first named author was supported by the Slovenian Research Agency (ARRS) Grant P1-0288. }


\maketitle

\section{Introduction}

In 1878, 35 years after Hamilton's discovery of the algebra of quaternions $\mathbb H$, Frobenius published 
the paper \cite{F} in which he
showed that 
there are no other finite-dimensional real division algebras than $\mathbb R$, $\mathbb C$, and $\mathbb H$.
The present paper is centered around this beautiful classical theorem.

Several proofs of Frobenius' Theorem are known. Some textbooks (for example,  \cite{INCA, Lam}) present elementary ones  and some (for example, \cite{Hun, P})
present
more advanced ones. In Section \ref{s2}, we give a new elementary proof which, we believe,
nicely illustrates the usefulness of linear algebra concepts and may  therefore be interesting also for students.

Two striking features of the algebra of quaternions, which are also reflected in our proof of Frobenius' Theorem, are  the existence of elements having square $-1$ and the existence of pairs of  nonzero (and hence invertible) anticommuting elements.
 In Section \ref{s3}, we consider such elements in an arbitrary finite-dimensional real algebra $A$.
We show in particular that $\mathbb C$ (resp. $\mathbb H$) embeds in $A$ if and only if the dimension of every left ideal of $A$ is an even number (resp. a multiple of $4$). In the case of $\mathbb C$, we can replace "left ideal" by  "ideal" in this statement. The condition that  the dimension of every ideal of $A$ is  a multiple of $4$, however, is equivalent to the condition that $A$ contains a pair of anticommuting invertible elements.

Our proofs in Section \ref{s3} depend on Wedderburn's Principal Theorem 
which is used to reduce  problems on general algebras to  semisimple ones.
 Motivated by the results on lifting idempotents and matrix units modulo ideals, we have asked ourselves whether some general lifting properties could be used instead of this theorem.
However, we discovered that lifting of the standard  basis of the quaternions modulo nil ideals is not always possible, so this question 
 turned out to be more subtle than first expected. Accordingly, we have focused on
the (more narrow in some aspects and broader in others)
 problem of lifting arbitrary algebraic elements. This is the central topic of Section \ref{s4}. 
We show that such a lifting modulo any nil ideal is possible 
 if and only if the polynomial whose root is  the element in question is separable. A similar statement---which, however, requires the involvement of 
the splitting field---holds for lifting modulo any ideal of a finite-dimensional algebra.

We  close this introduction with a word on conventions. 
We will work in the framework of associative  algebras with $1$. Accordingly, all subalgebras are supposed to contain  $1$,
all modules are supposed to be unital, and all homomorphisms are supposed to send $1$ to $1$.

\section{Frobenius' Theorem}\label{s2}

Let $A$ be a real algebra. For any $\lambda\in \mathbb R$, we write $\lambda 1\in A$ simply as $\lambda$, and in this way identify $\mathbb R1$ with $\mathbb R$. 
We write 
$$a\circ b= ab+ba$$
for  any   $a,b\in A$, and we write
span$\{a_1,\dots,a_n\}$ for the linear span of the elements $a_1,\dots,a_n\in A$.

In this and the next section, we  restrict our attention to the case where $A$ is finite-dimensional.
Then every element in $A$ is algebraic over $\mathbb R$, i.e.,
there exists a nonzero  polynomial $f(X)\in \mathbb R[X]$ 
such that
 $f(a)=0$. The Fundamental Theorem of Algebra states, in one of its forms, that $f(X)$ is a product of linear and quadratic polynomials in $\mathbb R[X]$. Therefore, the product of linear and quadratic polynomials in $a$ is equal to $0$. This fact  is one of the two tools that will be used in our proof of Frobenius' Theorem. The second tool is even more elementary: every linear operator from an odd-dimensional real vector space to itself has a real eigenvalue.

 Before proceeding to Frobenius' Theorem, we record 
a little observation from which this paper actually stemmed.
 Assume that  the dimension of $A$ is an odd number.  Since the rule  $x\mapsto ax$, where $a$ is an arbitrary element in $A$, defines a linear operator  from $A$ to $A$,  there exist a $\lambda\in\mathbb R$ and a nonzero $v\in A$ such that
  $(a-\lambda)v =av-\lambda v=0$.
Thus, $a-\lambda$ is a zero-divisor, unless $a\in \mathbb R$. In particular, if $A$ is not $1$-dimensional and hence isomorphic to $\mathbb R$, it cannot be a division algebra. Note that  the associative law has not been used here, so  this holds even for nonassociative algebras. Of course,  this is anything but new---it is  well-known that finite-dimensional not necessarily associative real division algebras exist only in dimensions $1$, $2$, $4$, and $8$ \cite{BM, Ker}. Our excuse for pointing out this argument
is that it may be used for pedagogical purposes, for example to motivate the introduction of the  algebra of quaternions $\mathbb H$ (and perhaps also the algebra of octonions $\mathbb O$). Let us recall,
for convenience of the reader,  that $\mathbb H$ is a 
$4$-dimensional real division algebra whose standard basis consists of $1$ and elements $i,j,k$ satisfying
$$i^2=j^2=k^2=-1,\,\,\, ij=-ji=k,\,\,\, jk=-kj =i,\,\,\,ki=-ik=j$$ ($\mathbb O$, on the other hand, is an $8$-dimensional  nonassociative real division algebra; note that when using the term algebra without any adjective, we always have in mind an associative algebra.)


 We now turn to Frobenius' Theorem.

\begin{theorem} {\bf (Frobenius' Theorem)} \label{f} A finite-dimensional real division algebra $D$ is isomorphic  to 
$\mathbb R$, $\mathbb C$, or $\mathbb H$.
\end{theorem}

\begin{proof}
We divide the proof into three steps. 

\smallskip
(a)  {\em  For every $a\in D$, there exist $\alpha,\beta\in \mathbb R$ such that $a^2 = \alpha a+\beta$. Moreover, if $\alpha=0$ and $a\notin \mathbb R$, then $\beta < 0$ and so $(\frac{1}{\sqrt{-\beta}}a)^2=-1$.}
\smallskip

{\em Proof of {\rm (a)}}.
The first statement is
a  consequence of the aforementioned version of  the Fundamental Theorem of Algebra. Suppose
 $\alpha= 0$. If $\beta \ge  0$, then$(a- \sqrt{\beta})(a+ \sqrt{\beta})=0$, yielding $a=\pm \sqrt{\beta}\in\mathbb R$.
{\tiny $\square$}
\smallskip

\smallskip 
(b) {\em Let $a\in D$ be such that $a^2=-1$. Then, for every $b\in D\setminus  {\rm span}\{1,a\}$, there exists a $c\in {\rm span}\{1,a,b\}\setminus  {\rm span}\{1,a\}$ such that $a\circ c =0$ and $c ^2=-1$.}
\smallskip

{\em Proof of {\rm (b)}}. 
Since $a\circ b = (a+b)^2-a^2-b^2$,  we see from (a) that
$a\circ b$ lies in $V={\rm span}\{1,a,b\}$. Consequently, the 
 linear operator $x\mapsto a\circ x$ maps  $V$ to itself. 
 Since $\dim V =3$, there exist a $\lambda\in\mathbb R$ and  a nonzero $v\in V$ such that
$a\circ v = \lambda v$. That is, $\lambda - a = vav^{-1}$, and hence $(\lambda -a)^2 = va^2v^{-1}=-1$.
On the other hand, $(\lambda -a)^2 = \lambda^2 - 2\lambda a -1$. Comparing the two identities we get $\lambda^2 - 2\lambda a =0$, which yields $\lambda=0$ since $a\notin \mathbb R$.
Thus,    $a\circ v =0$.
This implies that
$a$ commutes with $v^2$ but not with $v$. 
 From (a) it first follows that $v^2\in\mathbb R$, and then that there exists a $c\in \mathbb R v$ such that
$c^2=-1$. Of course,  $a\circ c=0$ and so  $c\notin  {\rm span}\{1,a\}$.
{\tiny $\square$}

\smallskip
(c) {\em $D$ is isomorphic to $\mathbb R$, $\mathbb C$, or $\mathbb H$.}
\smallskip

{\em Proof of {\rm (c)}}. 
We may assume that $\dim_{\mathbb R}  D> 1$ (since otherwise
 $D\cong\mathbb R$).   
As  $a^2-\alpha a\in \mathbb R$ implies 
 $(a-\frac{\alpha}{2})^2\in\mathbb R$, it follows from (a) that   $D$ contains an element $i$ satisfying $i^2=-1$.  Therefore, $D\cong \mathbb C$ if
$\dim_{\mathbb R}  D= 2$, so let $\dim_{\mathbb R}  D> 2$.
By (b),  there exists a $j\in D\setminus  {\rm span}\{1,i\}$ such that $i\circ j =0$ and
$j^2=-1$. Set $k=ij$. Note that $k ^2=-1$, $k i = -ik=j$, and $jk=-kj=i$.
 Using $k\circ i = k\circ j=0$ it is easy to see that $k\notin {\rm span}\{1,i,j\}$. Thus, $D$ contains 
a subalgebra isomorphic to $\mathbb H$. It remains to  show that this subalgebra is not proper.
 Suppose this were not true.  Then,
by (b), there would exist a $c\in D\setminus {\rm span}\{1,i,j,k\}$ such that 
$i\circ c =0$, implying that $jc$ commutes with $i$. However, only the elements in span$\{1,i\}$ commute with
 $i$---indeed, if $b\notin {\rm span}\{1,i\}$ then  $bi\ne ib$ since, by (b),  $ {\rm span}\{1,i,b\}$ contains an element not commuting with $i$. Thus, $jc = \gamma i + \delta$ for some $\gamma,\delta\in \mathbb R$.
Multiplying from left by $-j$ we get $c = \gamma k  - \delta j$, contradicting $c\notin {\rm span}\{1,i,j,k\}$.
{\tiny $\square$}
\end{proof}

It is clear from the proof that we may replace the finite-dimensionality of $D$ by a milder assumption that $D$ is an algebraic algebra.

\section{Quaternion-Like Elements in Finite-Dimensional Algebras}\label{s3}

Our proof of Frobenius' Theorem was based on elements whose square is $-1$, i.e., elements that generate an algebra  
isomorphic to $\mathbb C$. This was a natural approach since the algebra of quaternions $\mathbb H$ has plenty of such elements.  We will now show that the existence of only one such element in a finite-dimensional real algebra has an impact on its global structure. First we introduce some notation: by
$M_n(B)$ we denote the algebra of $n\times n$ matrices over the algebra $B$, and by rad$(A)$ 
the  radical of the finite-dimensional algebra $A$ (i.e., rad$(A)$ is the (unique) maximal  nilpotent ideal of $A$).

\begin{theorem}\label{even1}
Let $A$ be a nonzero finite-dimensional   real algebra. The following statements are equivalent:
\begin{enumerate}
\item[{\rm (i)}] $\mathbb C$ embeds in $A$ (i.e., $A$ contains an element $a$ satisfying  $a ^2=-1$).
\item[{\rm (ii)}] $A$ contains an element $a$ such that $a-\lambda$ is invertible for every $\lambda\in\mathbb R$.
\item[{\rm (iii)}] The dimension of any left ideal of  $A$ is an even number.
\item[{\rm (iv)}] The dimension of any ideal of  $A$ is an even number.
\end{enumerate}
\end{theorem}

\begin{proof} 
(i)$\implies$(ii). 
Note that $a^2=-1$ implies $(a-\lambda)(a+\lambda) = -(1+\lambda^2)$.

(ii)$\implies$(iii). Suppose $A$ has a left ideal $I$  of
 odd dimension. Since the linear operator $x\mapsto ax$ maps $I$ to $I$, 
it has a   real eigenvalue  $\lambda$. Thus, $(a-\lambda)v=0$ for some nonzero $v\in I$, showing that
$a-\lambda$ is not invertible.

(iii)$\implies$(iv). Trivial.

(iv)$\implies$(i).  By
Wedderburn's Principal Theorem  (see, e.g., \cite[p.\,191]{Row}),  $A$ is the vector space direct sum of  a semisimple subalgebra $A'$ (isomorphic to $A/{\rm rad}(A)$) and rad$(A)$. Wedderburn's structure theory (see, e.g., \cite[Theorem 2.64]{INCA}) along with Frobenius' theorem tells us that $A'$ is isomorphic to the direct product of (finitely many) simple algebras
$S_i$, each of which is isomorphic to one of the matrix algebras $M_{n}(\mathbb R)$,  $M_{n}(\mathbb C)$, and  $M_{n}(\mathbb H)$. It is enough to show that each $S_i$ contains an element whose square is $-1$ (here, of course, $1$ stands for the unity of $S_i$ rather than the unity of $A$). If $S_i\cong M_n(\mathbb C)$ or
 $S_i\cong M_{n}(\mathbb H)$, then we can simply take the scalar matrix with $i$  on the diagonal. Assume that
$S_i\cong M_n(\mathbb R)$. Note that $S_i \oplus {\rm rad}(A)$ is an ideal of $A$. Since its dimension, as well as the 
dimension of rad$(A)$, is an even number, it follows that so is the dimension of $S_i$. But then $n$ must be even too.
Therefore, $M_n(\mathbb R)$ contains the block diagonal matrix with blocks $\left[ \begin{smallmatrix} 0 & 1 \cr -1 & 0 \cr \end{smallmatrix} \right]$, which  has the desired property.
\end{proof}

\begin{remark}
By the Fundamental Theorem of Algebra, (i) can be equivalently stated as follows:  (i') For every $f(X)\in \mathbb R[X]$ there exists an $a\in A$ such that $f(a)=0$.
\end{remark}

\begin{remark} It is clear from the proof that
the implication (i)$\implies$ (iv)  holds in any nonassociative algebra. This is not true for the implication (iv)$\implies$(i).  Indeed, let $A$  be the $4$-dimensional real space  of self-adjoint $2\times 2$ complex matrices. Endowing $A$ with the product $x\circ y = xy+yx$, $A$ becomes a 
a simple Jordan algebra,  so it has no nontrivial (left) ideals. Thus, $A$ satisfies  (iv) (and (iii)). 
However, it does not  satisfy (i) (and (ii))  since
self-adjoint matrices have real eigenvalues.
\end{remark}

In the next theorem, we consider the situation where $A$ has elements $a,b$ satisfying $a^2=b^2=-1$ and $a\circ b=0$. As is evident from our proof of Frobenius' Theorem,
such a pair of elements generates a subalgebra isomorphic to $\mathbb H$.

\begin{theorem}\label{even3}
Let $A$ be a finite-dimensional  real algebra. The following statements are equivalent:
\begin{enumerate}
\item[{\rm (i)}] $\mathbb H$ embeds in $A$ (i.e., $A$ contains elements   $a,b$ satisfying $a^2=b^2=-1$ and $a\circ b=0$).
\item[{\rm (ii)}] The dimension of any left ideal of  $A$ is a multiple of $4$.
\end{enumerate}
\end{theorem}

\begin{proof} 
(i)$\implies$(ii). Since $A$ contains a subalgebra isomorphic to $\mathbb H$, every left ideal of $A$ is a finite-dimensional  left $\mathbb H$-module. We claim that the dimension  of any such module $M$ is a multiple of $4$. This follows  from the general theory of modules over simple rings, but let us rather give a short direct proof. We may assume that our claim is true for all modules having smaller dimension than $M$. Take a nonzero $m\in M$. Since $h\mapsto hm$ is a linear isomorphism from $\mathbb H$ onto the submodule $\mathbb H m$, the dimension of $\mathbb H m$ is $4$. By assumption, the dimension of 
the quotient module $M/\mathbb H m$ is a multiple of $4$. But then so is the dimension of $M$.


(ii)$\implies$(i).  Let $A'$ and  $S_i$ be as in the proof of (iv)$\implies$(i) of Theorem \ref{even1}. It is enough to show that each $S_i$ contains a pair of anticommuting elements whose square is $-1$.
If $L$ is a left ideal
of $S_i$, then $L\oplus {\rm rad}(A)$ is a left ideal of $A$. This readily implies that the dimension of $L$ is a multiple of $4$. Consequently,  $S_i$ is isomorphic either to  $M_{4k}(\mathbb R)$, $M_{2k}(\mathbb C)$, or $M_k(\mathbb H)$ for some $k\ge 1$.
Since $\mathbb H$ embeds in $M_4(\mathbb R)$ via the regular representation, and since the complex matrices 
of the form $\left[ \begin{smallmatrix} z & w \cr -\overline{w} & \overline{z} \cr \end{smallmatrix} \right]$ 
 form a subalgebra
of $M_2(\mathbb C)$ isomorphic to $\mathbb H$, in each of the threee cases $S_i$ contains elements with desired properties.
\end{proof}

We cannot substitute "ideal" for "left ideal" in (ii). Indeed, just think of the algebra $M_2(\mathbb R)$ which does not satisfy
the conditions of Theorem \ref{even3}, but its only  two  ideals are of dimensions $0$ and $4$.
 The next and last theorem of this section answers
the now obvious question of which algebras have only ideals whose dimensions are multiples of $4$.

\begin{theorem}\label{even2}
Let $A$ be a finite-dimensional real algebra. The following statements are equivalent:
\begin{enumerate}
\item[{\rm (i)}] $A$ contains invertible elements $u,v$ satisfying $u \circ v =0$.
\item[{\rm (ii)}] $A$ contains  elements $a,b$ satisfying $a^2=-1$, $b^4=1$, and $a\circ b=0$.
\item[{\rm (iii)}] The dimension of any ideal of  $A$ is a multiple of $4$.
\end{enumerate}
\end{theorem}

\begin{proof} 
(i)$\implies$(ii). Assume first that $A$ is semisimple. Then $A$ is isomorphic to the direct product of simple algebras
$S_i$, each of which satisfies (i).  It is enough to show that each of them satisfies (ii). 
If $S_i$ is isomorphic to $M_n(\mathbb H)$, then this is clear. Indeed,  take $a$ (resp. $b$) to be the scalar matrix with $i$ (resp. $j$)  on the diagonal, and observe that
 $a\circ b =0$ and  $a^2=b^2=-1$. 
We may thus assume that $S_i$ is isomorphic to  $M_n(\mathbb F)$ with $\mathbb F\in \{\mathbb R,\mathbb C\}$. From $uv = -vu$ we get $\det(u)\det(v) = (-1) ^n\det(v)\det(u)$, so $n$ must be an even number. Thus, we can  define $a$ to be the block diagonal matrix with blocks $\left[ \begin{smallmatrix} 0 & 1 \cr -1 & 0 \cr \end{smallmatrix} \right]$ and $b$
to be 
 the block diagonal matrix with blocks $\left[ \begin{smallmatrix} 0 & 1 \cr 1 & 0 \cr \end{smallmatrix} \right]$. Observe that $a\circ b =0$,  $a^2=-1$, and $b^2 =1$.

Let now $A$ be arbitrary. 
By Wedderburn's Principal Theorem, 
$A$ contains a subalgebra $A'$ isomorphic to  $A/{\rm rad}(A)$ such that
 $A=A'\oplus {\rm rad}(A)$. By the preceding paragraph, $A'$ contains a pair of elements satisfying (ii). But then  so does
$A$.

(ii)$\implies$(iii). It is enough to prove that the dimension of $A$ is a multiple of $4$. Indeed, if $I$ is an ideal of $A$,
 then $A/I$ also satisfies (ii). Thus, knowing that the dimensions of  $A$ and $A/I$ are  multiples of $4$, it follows that the same is true for $\dim I = \dim A - \dim A/I$.

We begin our proof that  the dimension of $A$ is a multiple of $4$ by noticing that $A_0 = \{x_0\in A\,|\, ax_0=x_0a\}$ is a subalgebra of $A$ containing all elements of the form $x-axa$ with $x\in A$, and $A_1 =\{x_1\in A\,|\, ax_1=-x_1a\}$ is an $A_0$-bimodule containing all elements of the form $x+axa$ with $x\in A$. Moreover, $A_1^2 \subseteq A_0$.
Since $A_0\cap A_1 = \{0\}$ and  $$x = \frac{1}{2}(x - axa) +  \frac{1}{2}(x + axa)$$
holds for every $x\in A$, it follows that
$A=A_0\oplus A_1$  (so $A_0$ and $A_1$ give rise to a $\mathbb Z_2$-grading of $A$). By assumption, $b\in A_1$ and so $bx_0\in A_1$ for every $x_0\in A_0$. Conversely, if $x_1\in A_1$, then we see from $b^{-1}=b^3\in A_1$ that $b^{-1}x_1\in A_0$, and so $x_1$ is of the form $bx_0$ with $x_0\in A_0$. Thus, $x_0\mapsto bx_0$ is a linear isomorphism from $A_0$ onto $A_1$, implying that $\dim A = \dim A_0 + \dim A_1 = 2\dim A_0$. Since $a\in A_0$ and 
 $a^2=-1$, Theorem \ref{even1} tells us that $\dim A_0$ is an even number. Consequently, $\dim A$ is a multiple of $4$.

(iii)$\implies$(i). 
By Wedderburn's Principal Theorem, $A=A'\oplus {\rm rad}(A)$ where $A'$ is a semisimple algebra. Let $S_i$
be simple algebras whose direct product is $A'$. Since  $S_i\oplus {\rm rad}(A)$
and rad$(A)$ are ideals of $A$ and so their dimensions  are multiples of $4$, the same holds for the dimension of $S_i$. But then $S_i$ is isomorphic either
 to $M_{2k}(\mathbb R)$, 
$M_{2k}(\mathbb C)$ or $M_{k}(\mathbb H)$ for some $k\ge 1$. In each case, we can find a pair of anticommuting invertible elements in $S_i$ (cf. the proof of (i)$\implies$(ii)). These elements obviously yield a  pair of anticommuting invertible elements in $A'\subseteq A$.
\end{proof}

\section{Polynomials Having Liftable Roots}\label{s4}

Throughout this section, $F$ will be a field and $A$ will be an algebra over $F$. 
Recall that an ideal $I$ of $A$ is said to be  nil if all of its elements are nilpotent. Nilpotent ideals 
(i.e., ideals $I$ such that $I^n=0$ for some $n\ge 1$)
 are obviously nil, but nil ideals may not be nilpotent.

We say that idempotents can be lifted modulo an ideal $I$ of $A$ if every idempotent $b$ in $A/I$ is of the form
$e+I$ where $e$ is an idempotent in $A$.  A classical result in ring theory states 
that idempotents can   be lifted modulo any  nil ideal  (see, e.g., \cite[Theorem 21.28]{Lam}). 
This means that the polynomial $X^2-X$ has the property described in 
 the following definition.

\begin{definition} A polynomial $f=f(X)\in F[X]$ has {\em nil-liftable roots over $F$} if, for every  $F$-algebra
$A$ and every nil ideal $I$ of $A$,  any element $b\in A$ satisfying 
$f(b)\in I$ is of the form $b = a + u$ where $u\in I$ and $a\in A$ is such that
$f(a)=0$.
\end{definition}

In other words,
$f$ has   nil-liftable roots if, for every nil ideal $I$ of any $F$-algebra $A$, the following holds:
 if $\overline{b}=b + I\in A/I$ is such that $f(\overline{b}) =0$, then there exists an $a\in A$ such that
$f(a) =0$ and $\overline{b} = \overline{a}\, (=a+I)$. 

Of course, not every idempotent can be lifted modulo an ideal. For example, the polynomial algebra $F[X]$ has no idempotents different from $0$ and $1$, but the quotient algebra $F[X]/\mathcal I$, where $\mathcal I$ is the principal ideal generated by 
$X^2-X$, has  an idempotent $X + \mathcal I$. 
On the other hand,
as a byproduct of the general theory of lifting idempotents \cite{N} we have that idempotents can  be lifted modulo any ideal of a finite-dimensional algebra. That is, the polynomial $X^2-X$ has also the following  
 property.

\begin{definition} A polynomial $f=f(X)\in F[X]$ has {\em liftable roots over $F$ in finite dimensions} if, for every finite-dimensional $F$-algebra
$A$ and every  ideal $I$ of $A$,  any element $b\in A$ satisfying 
$f(b)\in I$ is of the form $b = a + u$ where $u\in I$ and $a\in A$ is such that
$f(a)=0$.
\end{definition}

These two definitions suggest the problem of finding other algebraic elements, different from idempotents,  that can 
be lifted modulo certain ideals. This problem has been studied in some algebras of functional analysis---see \cite{Bar, Had, TS} and references therein. Among these papers, only 
\cite{Bar} seem to have  a (small) overlap with the results that we are about to establish.
On the other hand,  we were unable to find similar results in pure algebra. Perhaps we can mention the paper \cite{KLN} which considers the polynomial $X^2-1$, but discusses problems  of  a somewhat different nature. In fact,
if the characteristic of $F$ is different from  $2$, then
 the problem of 
lifting idempotents (i.e., roots of $X^2-X$) modulo $I$ is equivalent to the problem of lifting roots of $X^2-1$ modulo $I$.
This is because $e$ is an idempotent if and only if $1-2e$ has square $1$. The following simple example shows that
the polynomial $X^2+1$ is quite different with respect to the lifting problem.

\begin{example}\label{ecr}
 The real algebra $\mathbb C\times \mathbb R$ has no elements with square $-1$, but its homomorphic image $\mathbb C$ does. This shows that the polynomial $X^2 +1$ does not have  liftable roots over $\mathbb R$ in finite dimensions. 
\end{example}

However,  $X^2 +1$ has  liftable roots over $\mathbb C$ in finite dimensions. Indeed, since linear polynomials 
obviously have liftable roots (in finite or infinite dimensions), this follows from the next lemma.

\begin{lemma}\label{lv2}
If relatively prime polynomials $f_1,f_2\in F[X]$ have liftable roots over $F$ in finite dimensions, then so does their product  $f=f_1f_2$. 
\end{lemma}

\begin{proof}
Let $A$ be a   finite-dimensional $F$-algebra and let $I$ be an ideal of $A$.
Take $b\in A$ such that $f(b)=f_1(b)f_2(b)\in I$. Replacing, if necessary, $A$ by the subalgebra $B$ of $A$ generated by $b$ together with replacing $I$ by
$I\cap B$, we see that there is no loss of generality in assuming that $A$ is commutative.

Since $f_1$ and $f_2$ are relatively prime, there exist $g_1,g_2\in F[X]$ such that
$g_1f_1+g_2f_2 =1,$ and hence
 $$g_1(b)f_1(b)+g_2(b)f_2(b) =1.$$
Multiplying by $g_1(b)f_1(b)$, we get
 $$\big(g_1(b)f_1(b)\big)^2 - g_1(b)f_1(b) \in I.$$
Since  idempotents can be lifted modulo any ideal of a  finite-dimensional algebra \cite{N}, there exists an idempotent  $e_1\in A$ satisfying
$$e_1-g_1(b)f_1(b) \in I.$$ Setting $e_2 = 1-e_1$, we thus have
 $$e_2-g_2(b)f_2(b)\in I.$$
Hence,
 $$e_2f_1(b) = \big(e_2-g_2(b)f_2(b)\big) f_1(b) + g_2(b)f(b) \in I.$$
 Multiplying this relation by $e_2$, we see that
 $e_2f_1(b)\in e_2I.$
 Similarly, $e_1f_2(b)\in e_1I.$
  
  Let us now  consider the (finite-dimensional) algebra  $A_2 = e_2A$ and its ideal $I_2 = e_2I$. Since $e_2$ is the unity of
  $A_2$,
   $f_1(e_2b)$ is equal to $e_2f_1(b)$, and therefore belongs to $I_2$.  
Using the assumption that $f_1$ has liftable roots,  it follows that
 $e_2b = a_2+u_2$ where $u_2\in I_2$, $a_2\in A_2$, and $f_1(a_2) = 0$. Similarly, $e_1b = a_1 +u_1$ with
  $u_1\in I_1=e_1I$, $a_1\in A_1=e_1A$, and $f_2(a_1) = 0$. Set $u=u_1+u_2\in I$ and $a=a_1+a_2\in A$.
Then $b-a\in I$. Moreover, using
$e_1a_2 =e_2a_1=0$,  we see that
 \begin{align*}f(a)& = e_1f(a)+ e_2f(a) = e_1f(e_1a) + e_2f(e_2a)\\
 &= e_1f(e_1a_1) + e_2f(e_2a_2) = e_1f(a_1) + e_2f(a_2) =0,
\end{align*}
as desired.
\end{proof}

Essentially the same proof shows that this lemma also holds for nil-liftable roots. However, we will not need that; in fact, the next lemma will yield a result that tells us much more.

 Recall that a nonzero  polynomial $f\in F[X]$ is said to be  separable if  its roots in the splitting field are all distinct.
 Equivalently, $f$ and its derivative $f'$ are relatively prime in $F[X]$.

\begin{lemma}\label{ls}
A nonzero polynomial  $f\in F[X]$ is  separable  if and only if there exist polynomials $h,k\in F[X]$ such that
$$f\big(X - h(X)f(X)\big) = k(X)f(X)^2.$$
\end{lemma}

\begin{proof}   Using the binomial theorem, we see that 
 there exists a polynomial
$\widetilde{f}(X,Y)\in F[X,Y]$  such that
$$f(X+Y)= f(X) + f'(X)Y + \widetilde{f}(X,Y)Y^2.$$
Now suppose that $f$ is separable. Then $f$ and $f'$ are relatively prime, so $gf + hf'=1$ for some polynomials $g,h\in F[X]$. Hence,
\begin{align*}
&f\big(X - h(X)f(X)\big)\\
 =& 
 f(X) - f'(X) h(X)f(X) + \widetilde{f}\big(X, -h(X)f(X)\big) h(X)^2f(X)^2\\
=& \Big(g(X)   + \widetilde{f}\big(X, -h(X)f(X)\big) h(X)^2\Big)f(X)^2,
\end{align*}
which proves the "only if" part.

To prove the converse, assume that  polynomials $h$ and $k$ are as in the statement of the lemma. 
Suppose that $f$ is not separable. Then there exist an element $\beta$ in the splitting field $K$ for $f$ over $F$, a polynomial 
$q(X)\in K[X]$ such that $q(\beta)\ne 0$, and $k>1$ such that
$$f(X)=(X-\beta)^k q(X).$$
Accordingly,
$$\big(X - h(X)f(X) - \beta\big)^k q\big(X - h(X)f(X)\big) =k (X)(X-\beta)^{2k}q(X)^2.
$$
Since 
$$X - h(X)f(X) - \beta = (X-\beta)\Big(1- h(X)q(X)(X-\beta)^{k-1}\Big),$$
it follows that
$$\Big(1- h(X)q(X)(X-\beta)^{k-1}\Big)^k q\big(X - h(X)f(X)\big) = k (X)(X-\beta)^{k}q(X)^2.$$
Evaluating at $\beta$ gives $q(\beta)=0$, a contradiction.
\end{proof}

We will  need only the "only if" part of this lemma, which can be sharpened as follows.

\begin{lemma}\label{ls2} Let $f\in F[X]$ be a  separable polynomial.
 Then for every $n\ge 1$  there exist polynomials $h_n,k_n\in F[X]$ such that
$$f\big(X - h_n(X)f(X)\big) = k_n(X)f(X)^{2^n}.$$
\end{lemma}

\begin{proof}
 Lemma \ref{ls} covers the case where $n=1$. Assume, therefore, that $h_i,k_i$, $i=1,\dots,n$, exist, and let us 
 prove the existence of $h_{n+1}$ and $k_{n+1}$. Set $r_i(X)= X - h_i(X)f(X)$, $i=1,\dots,n$. We have
\begin{align*}
f(r_n(r_1(X))) = k_n(r_1(X))f(r_1(X))^{2^n}= k_{n+1}(X) f(X)^{2^{n+1}}
\end{align*}
where $k_{n+1}(X)=k_n(r_1(X))k_1(X)^{2^n}$, 
and
$$r_n(r_1(X)) = r_1(X) - h_n(r_1(X))f(r_1(X)) = X- h_{n+1}(X)f(X),$$
where $h_{n+1}(X) =
 h_1(X)+ h_n(r_1(X))k_1(X)f(X)$.
\end{proof}

We are now  in a position to prove the main theorem of this section.

\begin{theorem}\label{tc}
Let $F$ be a field, let $f\in F[X]$ be a nonzero polynomial, and let $K$ be the splitting field for $f$ over $F$. The following statements are equivalent:
\begin{enumerate}
\item[{\rm (i)}] 
 $f$ is separable.
\item[{\rm (ii)}] 
 $f$ has nil-liftable roots over $F$.
\item[{\rm (iii)}] 
 $f$ has liftable roots over $K$ in finite dimensions.
\end{enumerate}
\end{theorem}

\begin{proof}
(i)$\implies$(ii). This follows from Lemma \ref{ls2}.
Indeed, let $A$ be an $F$-algebra with nil ideal $I$, and let  $b\in A$ be such that $f(b) \in I$. 
Then $f(b)^{2^n}=0$ for some $n\ge 1$, and so $a= b -h_n(b)f(b)$ satisfies $f(a)=0$ and 
$a-b\in I$.

(i)$\implies$(iii). Since linear polynomials obviously have nil-liftable roots, this follows from  Lemma \ref{lv2}. 

(ii)$\implies$(i) and (iii)$\implies$(i).
Suppose  $f$ is not separable. Then
$$f(X)= \lambda(X-\beta_1)^{k_1}(X-\beta_2)^{k_2}\cdots (X-\beta_r)^{k_r}$$ where 
$\beta_i\in K$ are the distinct roots of $f$, $\lambda$ is the leading coefficient of $f$,
 and $k_i$ are positive integers such that at least one of them, say $k_1$, is greater than $1$. Let $A$ be the algebra of 
all $(k_1+1)  \times (k_1+1)$ upper triangular matrices over $K$; we may view $A$ as a (finite-dimensional) algebra over $K$ as well as over $F$.
 As usual, we denote by $e_{ij}$ 
the standard matrix units. 
Let $I$ be the ideal consisting of matrices of the form  $\gamma e_{1,k_1+1}$ with $\gamma\in K$.
Observe that $tu = ut=0$ for every $u \in I$ and every strictly upper triangular matrix $t\in A$. In particular, $I^2= 0$.
Set 
$$b = \beta_1 + e_{12} + e_{23}+\dots + e_{k_1,k_1+1}.$$
Then $(b-\beta_1)^{k_1} = e_{1,k_1+1}$ and hence 
$$f(b) = \lambda(\beta_1-\beta_2)^{k_2}\cdots (\beta_1-\beta_r)^{k_r} e_{1,k_1+1}\in I.$$
It is easy to see  that $f(b+y) = f(b)\ne 0$ for every $y\in I$, which shows that neither 
(ii) nor (iii) can hold.
\end{proof}



\begin{corollary}\label{tc2}
If a  nonzero polynomial  $f\in F[X]$ splits in $F$, then $f$ 
 has liftable roots over $F$ in finite dimensions if and only if $f$ is separable.
\end{corollary}


The proofs of Lemmas  \ref{ls} and \ref{ls2} describe
an algorithm for lifting modulo nil ideals, which, however, 
may be rather complicated when applied to concrete situations.
The following example provides an additional insight.

\begin{example}\label{r1}
First we record an  elementary remark  which is perhaps  of general interest.
Let $A$ be an $F$-algebra  and let $w\in A$ be such that $w^s=0$ for some $s\ge 1$. 
Then $1+w$ is invertible and
 $(1+w)^{-1}=\sum_{n=0}^\infty (-1)^n w^n$ (note that this sum is actually finite). Similarly, having in mind the binomial series for $(1+x)^{\frac{k}{m}}$ where $k,m\in\mathbb Z$, we see that $z=\sum_{n=0}^\infty {\frac{k}{m}\choose n}w^n$ satisfies
$z^m = (1+w)^{k}$; of course, this makes sense only under some restrictions on  the characteristic of $F$.

Now, consider  the polynomial $f(X)=X ^m -\beta$ where $\beta$ is a nonzero element in $F$. Let 
$I$ be a nil ideal of $A$ and let $b\in A$ be such that $f(b)\in I$. Thus,
there is a $w\in I$ such that $b^m = \beta(1+w)$. By the preceding paragraph, 
$A$ contains an element $z\in A$ such that $z-1\in {\rm span}\{w,w^2,\dots\}$ and
$z^m = (1+w)^{-1}$. Since $w$ commutes with $b$, so does $z$. Note that $a = z b$ satisfies
$a-b\in I$ and $a^m  =  \beta$, i.e., $f(a)=0$.
\end{example}

The next example concerns the classical case where $f=X^2 -X$. 

\begin{example}
Take  any  element $b$ from a ring $A$ such that $(b^2 - b)^k=0$ for some $k > 1$.
A direct computation shows that 
$$c= b - (2b-1)(b^2-b)=3b^2 - 2b^3$$
satisfies 
$$c^2-c = (b^2-b)^2 (4b^2-4b-3)$$
and so $(c^2-c)^{k-1}=0$. Hence, by induction on $k$ we see that  there exists an idempotent $e\in A$ such that $e-b \in  (b^2-b)\mathbb Z[b]$. This proves that idempotents can be lifted modulo nil ideals.

The proof just given  is extremely short, but may seem  ad hoc at first glance. However, from the  proof of Lemma \ref{ls}     we see that there is a concept behind it. The definition of $c$ is based on the observation that
$-4f(X) + (2X-1)f'(X)=1$ holds for $f(X)=X^2-X$.
\end{example}

As mentioned in the introduction, our original motivation for studying nil-liftable roots arose from the results of Section \ref{s3}. Let us conclude the paper by some remarks on the  liftable root problem in the context of that section. We first state a simple corollary of Theorem \ref{tc}.

\begin{corollary}\label{tbelow}
The polynomial $X^2+1\in F[X]$ has nil-liftable roots over $F$ if and only if the characteristic of $F$ is not $2$.
\end{corollary}

\begin{proof} If the characteristic of $F$ is $2$, then $f(X)=X ^2+1$ is equal to $(X+1)^2$, so it is not separable. If the characteristic of $F$ is not $2$, then $f$ cannot be equal to $(X+\gamma)^2 = X^2 + 2\gamma X + \gamma^2$, so it is separable.
\end{proof}

Combining this corollary with Theorem \ref{even1}  we obtain the following result.

\begin{corollary}\label{tbelow2}
If   a  (not necessarily finite-dimensional) real algebra $A$  contains a nil ideal $I$ of finite codimension and  an element $a$ such that 
$a-\lambda$ is invertible for every $\lambda\in \mathbb R$, then $\mathbb C$ embeds in $A$.
\end{corollary}

It should be mentioned that this result does not hold for any real algebra. Indeed, 
 every non-scalar element $a$ in the  algebra of rational functions $\mathbb R(X)$ has the property that $a-\lambda$ is invertible for every $\lambda\in \mathbb R$, but no element in  $\mathbb R(X)$ has square $-1$.

For $F=\mathbb R$, we can state Corollary \ref{tbelow}  as follows: If $I$ is a nil ideal of a real algebra $A$ and
$\varphi$ is a homomorhism from $\mathbb C$ to $A/I$, then there exists a homomorphism  $\psi:\mathbb C\to A$ such that $\pi\psi=\varphi$;
here, $\pi$ is the canonical epimorphism from $A$ to $A/I$. 
Since matrix units can be lifted modulo a nil ideal \cite[Proposition 13.13]{Row}, we can replace $\mathbb C$ with the algebra of $n\times n$ matrices in this statement. Perhaps surprisingly, the next example shows that this does not hold for the quaternion algebra $\mathbb H$.

\begin{example}\label{j}
Let $\mathcal F = \mathbb R\langle X,Y \rangle$ be the free algebra in two indeterminates $X,Y$, and let $\mathcal I$ be the ideal of $\mathcal F$ generated by 
$X^2+1$, $Y^2 +1$, and all elements of the form $(X\circ Y)f (X\circ Y)$ with $f\in \mathcal F$. Set $A = \mathcal F/\mathcal I$. It is an elementary exercise to show that $A$ is  $8$-dimensional.
 Obviously, $x= X + \mathcal I\in A$ and $y=Y+\mathcal I\in A$ satisfy $x^2=y^2=-1$
and the ideal $I$ of $A$ generated by $x\circ y$ has square $0$.

Clearly, $i=x+ I\in A/I$ and $j=y+I\in A/I$ satisfy
$i^2 = j^2=-1$ and $i\circ j=0$.
Since $A$ is generated by $x$ and $y$, $A/I$ is generated by $i$ and $j$. Therefore,
$A/I\cong \mathbb H$. This implies that $I$ is  maximal among nilpotent ideals of $A$, so $I = {\rm rad}(A)$.
Wedderburn's Principal Theorem therefore tells us that there must exist a subalgebra $A'$ of $A$ such that
$A'\cong \mathbb H$; and indeed,
 $x$ and $\frac{1}{2}(yx - xy)$ anticommute and their squares are  $-1$, 
so they generate a subalgebra isomorphic to $\mathbb H$ (we mention that
 $\big(\frac{1}{2}(yx - xy)\big)^2 =-1$ follows from $(x\circ y)^2 =0$).


However, we claim that no elements $a,b\in A$ satisfying  $a^2 = b^2 = -1$ and $a\circ b=0$ have the property that
$a-x, b- y\in I$. Indeed, suppose this were not true. Then $u=a-x\in I$ and $v=b-y\in I$ would satisfy  
 $(x+u)\circ (y+v)=0$. Since $I^2=0$, this gives
$$x\circ y  + u\circ y  + x\circ v =0.$$
As $u,v\in I$, $u=\sum_i p_i(x\circ y)q_i$ and 
 $v=\sum_i r_j(x\circ y)s_j$ 
for some $p_i,q_i,r_j,s_j\in A$. We can further write
$$p_i = f_i + \mathcal I,\,\,\, q_i = g_i + \mathcal I, \,\,\, r_j = h_j + \mathcal I, 
\,\,\, s_j = k_j + \mathcal I, $$
where $f_i,g_i,h_j,k_j\in \mathcal F$. Hence,
\begin{align*} X\circ Y+ \sum_i \big(f_i(X\circ Y)g_i\big)\circ Y + \sum_j X\circ \big(h_j(X\circ Y)k_j\big) 
 \in\mathcal I.\end{align*}
Note that each summand in the two summations consist of   monomials of degree at least $3$. 
Hence,    $X\circ Y$  must belong to  span$\{X ^2+1,Y^2+1\}$, which is a contradiction.
\end{example}

Wedderburn's Principal Theorem holds for all finite-dimensional algebras over perfect fields.
In light of Example \ref{j}, let us point out that
if $I$ is a nilpotent ideal of a finite-dimensional algebra $A$  and $\varphi$ is an embedding of a semisimple algebra $S$ to $A/I$,  then, under the  assumption that $F$ is perfect,  there is an embedding $\psi: S \to A$  such that $\pi(\psi(S)) = \varphi(S)$ (as above, $\pi$ stands for the canonical epimorphism). Indeed,
it is easy to see that $I$ is the radical of the algebra $A_1=\pi^{-1}(\varphi(S))$ and that $A_1/I\cong \varphi(S)$, so by Wedderburn's Principal Theorem
 there is an embedding $\psi$ of $S$ to $A_1\subseteq A$ for which $\pi(\psi(S)) = \varphi(S)$ obviously holds.
Such a weak form of lifting is thus possible, but Example \ref{j}
shows that we cannot always choose $\psi$ so that 
$\pi\psi=\varphi $---not even when $I={\rm rad}(A)$ and $S=A/I$. The following example  shows that if $F$ is not perfect, even such  "weak lifting" does not need to exist.

\begin{example}
If a field $F$ is imperfect, then it has prime characteristic $p$ and  contains  an element $a$ which is not of the form
$x^p$ with $x\in F$. Let $\mathcal I$ be the principal ideal of $F[X]$ generated by $(X ^p-a)^p =X^{p^2}- a^p$, and let $A=F[X]/\mathcal I$. Obviously, $x=X+\mathcal I$ satisfies $x^{p^2} = a ^p$
and $x^p -a\in{\rm rad}(A)$.
Therefore, the element $x+{\rm rad}(A)\in A/{\rm rad}(A)$  satisfies $(x+{\rm rad}(A))^p =a$. It is easy to see that $A$ does not contain an element whose $p$th power is $a$, so $A/{\rm rad}(A)$ does not embed in $A$.
\end{example}

As a final comment we mention that Corollary \ref{tbelow} yields a slightly different proof of Theorem \ref{even1} which avoids using Wedderburn's Principal Theorem. 
 On the other hand, Example \ref{j}  shows that it 
may be difficult to prove Theorem \ref{even3} without applying this old, deep result by Wedderburn.

\end{document}